\documentclass{article}
\usepackage{amsmath,amsthm,amssymb,amscd}

\newcommand{\range}{\mathrm{Range}}

\newcommand{\norm}[1]{\left\lVert #1 \right\rVert}

\def\restrict{\upharpoonright}

\newcommand{\bbR}{\mathbb{R}}

\newcommand{\dom}{\mathrm{Dom}}

\newtheorem{theorem}{Theorem}[section]
\newtheorem{lemma}[theorem]{Lemma}

\theoremstyle{definition}

\title{L\'{e}vy-Steinitz for countable sets of series}

\author{
Paul Larson
\thanks{Partially supported by NSF grant DMS-1201494. We thank Vladimir Kadets for brining Troyanski's paper to our attention.}
\\
Miami University\\
}

\begin{document}

\maketitle

\begin{abstract}
  The L\'{e}vy-Steinitz theorem characterizes the values that a conditionally convergent sequence in $\bbR^{n}$ can attain under permutations. We use material from \cite{rr} to extend this analysis to sequences in $\bbR^{\omega}$, under pointwise convergence, reproving a theorem of Stanimir Troyanski \cite{Troyanski}. 
\end{abstract}

It is shown in \cite{rr} that there exists a c.c.c. partial order adding a permutation of $\omega$ making every conditionally convergent real series in the ground model converge to a value not in the ground model. Applying this forcing fact to a countable elementary submodel of a sufficiently large fragment of the universe, one gets the following fact : for any countable set $S$ of conditionally convergent real series, and every countable $X \subseteq \bbR$, there is a permutation of $\omega$ making each member of $S$ converge to a real number not in $X$. In this note we give a more direct proof of this fact, using the same machinery. The resulting theorem (due to Stanimir Troyanski \cite{Troyanski}) is an extension of the L\'{e}vy-Steinitz theorem (which characterizes the values that a finite set of series can take under permutations) to countable sets of series. The proof uses the original L\'{e}vy-Steinitz theorem, as well as the Polygonal Refinement Theorem, which is used in the original proof of L\'{e}vy-Steinitz theorem.  It is a simplified version of the proof of the theorem from \cite{rr} mentioned above.

We emphasize that our extension of the L\'{e}vy-Steinitz theorem  applies to $\bbR^{\omega}$ under pointwise converge. Corollary 7.2.2 of \cite{KadetsKadets} says that in each infinite-dimensional Banach space there is a series attaining exactly two values under rearrangements.

\section{Preliminaries}\label{prelimsec}

We start with some material taken from \cite{Rosenthal}, as rewritten in \cite{rr}.

Given a sequence $\bar{a} = \langle a^{i} : i < d \rangle$ consisting of real-values series (for some $d \in \omega$), we let
$K(\bar{a})$ be the set of $\langle s_{i} : i < d\rangle \in \bbR^{d}$ for which the series
$\sum_{i\in d}s_{i}a^{i}$ is absolutely convergent. We let $R(\bar{a})$ be the orthogonal complement of $K(\bar{a})$ (i.e., the set of
vectors in $\bbR^{d}$ orthogonal to every element of $K(\bar{a})$). The sets $K(\bar{a})$ and $R(\bar{a})$ are each linear subspaces of $\bbR^{d}$,
and their dimensions sum to $d$.

We say that a set $I$ consisting of conditionally convergent series is \emph{independent} if $K(\bar{a}) = \{\bf{0}\}$ for each finite sequence $\bar{a}$ from $I$. We let $S(\bar{a})$ be the set of values in $\bbR^{d}$ of the form $\sum_{n \in \omega} \langle a^{i}_{p(n)} : i < d \rangle$ for $p$ a permutation of $\omega$.

The following theorem from 1913 is due to L\'{e}vy and Steinitz (see \cite{BonetDefant, KadetsKadets, Rosenthal}).

\begin{theorem}[L\'{e}vy-Steinitz]
If $\bar{a} = \langle a^{i} : i < d \rangle$ is a finite sequence of conditionally convergent real-valued series,
then \[S(\bar{a}) = \{ \langle \sum a^{i} : i < d \rangle + \bar{x} : \bar{x} \in R(\bar{a})\}.\]
\end{theorem}

One way to interpret the L\'{e}vy-Steinitz Theorem is to note that in the case where $\langle a^{i} : i < d \rangle$ is independent, it says that every value in $\bbR^{n}$ is attainable under some rearragement. If $\langle a^{i} : i < d \rangle$ is an arbitrary sequence of conditionally convergent real series,
then there exists a set $s \subseteq d$ such that $\{ a^{i} : i \in s \}$ is independent, and such that, for each $j \in d \setminus s$
there exist scalars $k_{i}$ $(i \in s)$, not all $0$ such that $\sum_{i \in s}k_{i}a^{i} + a^{j}$ is absolutely convergent. Given any permutation of $\omega$, then, the value of such an $a^{j}$ is determined by the values of $a^{i}$ $(i \in s)$. We carry out the version of this analysis for countable sets in the final section of this paper.

The following is the key lemma in the proof of the L\'{e}vy-Steinitz theorem (see \cite{Cohen, Rosenthal}).

\begin{theorem}[The Polygonal Confinement Theorem; Steinitz]
For each positive integer $d$ there
exists a constant $C_{d}$ such that for each positive $n \in \omega$ and all vectors $v_{m}$ ($m \in n$) from
$\bbR^{d}$,
if
\[\sum_{m \in n} v_{m} = 0\] and
$\norm{v_{m}} \leq 1$ for all $m \in n$,
then there is a permutation $p$ of $n \setminus \{0\}$ such that
\[\norm{
v_{0} +
\sum_{m \in k \setminus \{0\}}
v_{p(m)}
}\leq C_{d}\]
for every $m \in n + 1$.
\end{theorem}

The following immediate (and standard) consequence of the Polygonal Confinement Theorem is proved in \cite{rr}.

\begin{lemma}\label{pctrest}
Let $m$ and $d$ be positive integers, let $\rho$ be a positive real number and let $b$ and $v_{i}$ ($i \in m$) be elements of
$\bbR^{d}$.
Suppose that
\[\sum_{i \in m} v_{i} = b,\]
$\norm{b} \leq \rho$ and $\norm{v_{i}} \leq \rho$ for all $i \in m$.
Then there is a permutation $p$ of $m \setminus \{0\}$ such that
\[\norm{
v_{0} +
\sum_{i \in j \setminus \{0\}}
v_{p(i)}
}\leq \rho C_{d} + \norm{b}\]
for every $j \in m + 1$.
\end{lemma}

\section{Countable independent sets}

We prove in this section the version of the L\'{e}vy-Steinitz theorem for countable independent sets. The proof is an adaptation of arguments from \cite{rr}. The general version is proved in the next section.

\begin{theorem}\label{indthrm} Let $\langle a^{i} : i < \omega \rangle$ be an independent sequence of conditionally convergent real series and let
$\langle x_{i} : i < \omega \rangle$ be a sequence of real numbers. Then there is a permutation $p$ of $\omega$ such that, for each $i \in \omega$, $\sum_{j \in \omega} a^{i}_{p(j)} = x_{i}$.
\end{theorem}


For the rest of this section, fix $\langle a^{i} : i < \omega \rangle$ and $\langle x_{i} : i < \omega \rangle$ as in the statement of Theorem \ref{indthrm}, and a nondecreasing sequence of constants $C_{d}$ as given by the Polygonal Confinement Theorem.
We define a partial order $P$ from which our desired permutation will be induced by a suitable descending sequence.
Conditions in $P$ are triples $(f,d,\epsilon)$ such that
\begin{itemize}
\item $f$ is an injection from some $n \in \omega$ to $\omega$;
\item $d$ is a positive integer;
\item $\epsilon$ is a positive rational number;
\item $\norm{\sum_{k < n} \langle a^{i}_{f(k)} : i < d \rangle - \langle x_{i} : i < d \rangle} < \epsilon$;
\item for all $m \in \omega \setminus \range(f)$, $\norm{\langle a^{i}_{m} : i < d \rangle} < \epsilon/C_{d}$.
\end{itemize}


The order on $P_{I}$ is defined by : $ (g, e, \delta) \leq (f, d, \epsilon)$ if
\begin{itemize}
\item $g$ extends $f$;
\item $e \geq d$;
\item for all $m \in \dom(g) + 1$, $\norm{\sum_{k \in m \setminus \dom(f)} \langle a^{i}_{g(k)} : i < d \rangle} < 2\epsilon$;
\item $2\delta + \norm{\sum_{k \in \dom(g) \setminus \dom(f)} \langle a^{i}_{g(k)} : i < d \rangle} \leq 2\epsilon$.
\end{itemize}

Observe that if $\epsilon$ is greater than both $|x_{0}|$ and $|\sup \{ C_{1}a^{0}_{m} : m \in \omega\}|$ , then $(\emptyset, 1, \epsilon)$ is a condition in $P$.  If $\langle (f_{n}, d_{n}, \epsilon_{n}) : n \in \omega \rangle$ is a descending sequence in $P$ such that
\begin{itemize}
\item $\bigcup_{n \in \omega} f_{n}$ is a permutation of $\omega$,
\item $\omega = \bigcup_{n \in \omega} d_{n}$ and
\item $\lim_{n \to \infty}\epsilon_{n} = 0$
\end{itemize}
then $\bigcup_{n\in \omega} f_{n}$ is as desired. Theorem \ref{indthrm} follows then from Lemma \ref{twodense}.

\begin{lemma}\label{twodense} For each $(f, d, \epsilon) \in P$ and each $n \in \omega$, there exists a condition
$(g, d+1, \delta) \leq (f, d, \epsilon)$ with $n \subseteq \dom(g) \cap \range(g)$ and $\delta < 1/n$.
\end{lemma}

\begin{proof}
Let $(f, d, \epsilon)$ and $n$ be given.
By the L\'{e}vy-Steinitz theorem, there is a permutation $p$ of $\omega$ extending $f$ such that
\[\sum_{n \in \omega} \langle a^{i}_{p(n)} : i < d + 1\rangle
= \langle x_{i} : i < d + 1\rangle.\]
Let $\eta < \epsilon$ be such that $\norm{\langle a^{i}_{m} : i < d \rangle} < \eta/C_{d}$ for all $m \in \omega \setminus \dom(f)$, and
let $\delta \in \mathbb{Q}^{+}$ be smaller than both $1/n$ and $(\epsilon - \eta)/2$.
Fix $n_{*} \geq n$ such that
\begin{itemize}
\item $n \subseteq \range(p \restrict n_{*})$,
\item $\norm{\sum_{m \in n_{*} \setminus \dom(f)}\langle a^{i}_{p(m)} : i < d \rangle} < \epsilon$,
\item $\norm{\sum_{m \in n_{*}} \langle a^{i}_{p(m)} : i < d + 1\rangle - \langle x_{i} : i < d + 1 \rangle} < \delta$ and
\item $\norm{\langle a^{i}_{m} : i < d + 1 \rangle} < \delta/C_{d + 1}$ for each $m \in \omega \setminus n_{*}$.
\end{itemize}
By Lemma \ref{pctrest}, there is
an injection $g$ from $n_{*}$ to $\omega$ extending $f$, with the same range as $p \restrict n_{*}$, such that
\begin{equation*}
\begin{split}
\norm{
\langle a^{i}_{g(\dom(f))} : i < d \rangle +
\sum_{k \in m \setminus (\dom(f) + 1)}
\langle a^{i}_{g(k)} : i < d \rangle
} & \leq (\eta/C_{d}) C_{d} + (\epsilon - \delta) \\
 & < 2\epsilon - 2\delta
\end{split}
\end{equation*}
for every $m \in n_{*} + 1$. Then $(g,A, \delta)$ is as desired.
\end{proof}

\section{Arbitrary sequences}

We adapt the notation introduced in Section \ref{prelimsec} to countable sequences.
Given a sequence $\bar{a} = \langle a^{i} : i < \omega \rangle$ consisting of real-values series, we let
$K(\bar{a})$ be the set of $\langle s_{i} : i < \omega\rangle \in \bbR^{\omega}$ for which the following hold:
\begin{itemize}
\item the set $\{ i \in \omega : d_{i} \neq 0\}$ is finite;
\item the series $\sum_{i\in \omega}d_{i}a^{i}$ is absolutely convergent.
\end{itemize}
We let $R(\bar{a})$ be the orthogonal complement of $K(\bar{a})$ (i.e., the set of
vectors in $\bbR^{\omega}$ orthogonal to every element of $K(\bar{a})$). The sets $K(\bar{a})$ and $R(\bar{a})$ are each linear subspaces of $\bbR^{\omega}$. We let $S(\bar{a})$ be the set of values in $\bbR^{\omega}$ of the form $\sum_{n \in \omega} \langle a^{i}_{p(n)} : i < d \rangle$ for $p$ a permutation of $\omega$.

The following natural generalization of the L\'{e}vy-Steinitz theorem to countable sequences was first proved by Troyasnki \cite{Troyanski}.

\begin{theorem}[L\'{e}vy-Steinitz for countable sets]
If $\bar{a} = \langle a^{i} : i < \omega \rangle$ is a sequence of conditionally convergent real-valued series,
then \[S(\bar{a}) = \{ \langle \sum a^{i} : i < \omega \rangle + \bar{x} : \bar{x} \in R(\bar{a})\}.\]
\end{theorem}

\begin{proof}
  For each $i \in \omega$, let $s_{i} = \sum a^{i}$.
  Let $I \subseteq \omega$ be such that $\{ a^{i} : i \in I\}$ is independent, and such that, for each $j \in \omega \setminus i$ there exist
  $c_{j} \in \bbR$ and $d^{j}_{k} \in \bbR$ $(k \in I \cap j)$ such that $\sum_{k \in I \cap j} d^{j}_{k}a^{k} + a^{j}$ is absolutely convergent, with sum $c_{j}$.

  For one direction of the desired equality, let $\langle x_{i} : i < \omega \rangle$ be in $R(\bar{a})$. We want to find a permutation $p$ of $\omega$ such that, for each $i \in \omega$, $\sum_{n \in \omega}a^{i}_{p(n)} = s_{i} + x_{i}$. By Theorem \ref{indthrm}, there is a permutation $p$ such that this equation holds for all $i \in I$. Suppose now that $j$ is in $\omega \setminus I$. Since $\langle x_{i} : i \in \omega\rangle$ is in $R(\bar{a})$, $\sum_{k \in I \cap j} x_{k}d^{j}_{k} + x_{j} = 0$. Since $\sum_{k \in I \cap j} d^{j}_{k}a^{k} + a^{j}$ is absolutely convergent with sum $c_{j}$, $\sum_{k \in I \cap j} d^{j}_{k}s_{k} + s_{j} = c_{j}$ and
\begin{equation*}
\begin{split}
\sum_{n \in \omega}a^{j}_{p(n)} & = c_{j} - \sum_{k \in I \cap j} d^{j}_{k}\sum_{n \in \omega}a^{k}_{p(n)}\\
& = c_{j} - \sum_{k \in I \cap j}d^{j}_{k}(s_{k} + x_{k}) \\
& = (c_{j} - \sum_{k \in I \cap j} d^{j}_{k}s_{k}) - \sum_{k \in I \cap j}d^{j}_{k}x_{k}\\
&  = s_{j} + x_{j}
\end{split}
\end{equation*}
as desired.

 For the other direction, let $p$ be a permutation of $\omega$ such that $\sum_{n \in \omega} a^{i}_{p(n)}$ converges for all $n \in \omega$. For each $i \in \omega$, let
  $x_{i} = \sum_{n \in \omega} a^{i}_{p(n)} - s_{i}$. We want to see that $\bar{x} = \langle x_{i} : i \in \omega \rangle$ is in $R(\bar{a})$. To do this, fix $\langle d_{i} : i < \omega \rangle$ in $\bbR^{\omega}$ such that $\{ i \in \omega : d_{i} \neq 0\}$ is a finite set $D$, and such that
  $\sum_{i \in D} d_{i}a^{i}$ is absolutely convergent with sum $e$. Then
  \begin{equation*}
\begin{split}
\langle x_{i} : i < \omega \rangle \cdot \langle d_{i} : i < \omega \rangle & = \sum_{i \in D} x_{i}d_{i} \\
 & = \sum_{i \in D}d_{i}(\sum_{n \in \omega} a^{i}_{p(n)} - \sum_{n\in \omega}a^{i}_{n}) \\
  & = (\sum_{n \in \omega} \sum_{i \in D}d_{i}a^{i}_{p(n)}) - (\sum_{n\in \omega}\sum_{i \in D}d_{i}a^{i}_{n}) \\
 & = e - e\\
 & = 0.
\end{split}
\end{equation*}
\end{proof}


\begin{thebibliography}{99}

\bibitem{BonetDefant}
J. Bonet, A. Defant,
\emph{The L\'{e}vy-Steinitz rearrangement theorem for duals of metrizable spaces},
Israel J. Math. 117 (2000), 131–156


\bibitem{Cohen}
M.P. Cohen,
\emph{The descriptive complexity of series rearrangements}, Real Anal. Exchange 38 (2012/13), no. 2, 337–352




\bibitem{rr}
A. Blass, J. Brendle, W. Brian,  J.D. Hamkins, M. Hardy, P.B. Larson,
\emph{The rearrangement number},
in preparation

\bibitem{KadetsKadets}
M.I. Kadets, V.M. Kadets,
{\bf Series in Banach Spaces},
Birkhäuser Verlag, Basel, 1997

\bibitem{Rosenthal}
P. Rosenthal,
\emph{The remarkable theorem of L\'{e}vy and Steinitz},
Amer. Math. Monthly 94 (1987), no. 4, 342–351

\bibitem{Troyanski}
S. Troyanski, \emph{Conditionally converging series and certain F-spaces} (Russian),
Teor. Funkts., Funkts. Anal. Prilozh. 5, 102-107 (1967)

\end{thebibliography}

\end{document}